\documentclass[12pt,psamsfonts]{amsart}
\usepackage{amssymb}
\usepackage{amsmath}
\usepackage{graphicx}
\usepackage{amscd}
\usepackage{amsfonts}
\usepackage{amsbsy}
\usepackage{ifthen}
\usepackage[T1]{fontenc}
\usepackage[english]{babel}
\newtheorem {theorem} {Theorem}

\newtheorem{corollary}[theorem]{Corollary}
\newtheorem {remark} [theorem] {Remark}

\newtheorem {problem} [theorem] {Problem}
\newtheorem{question}[theorem]{Question}

\newcommand{\showcomments}{yes}

\newsavebox{\commentbox}
\newenvironment{comment}%
{\ifthenelse{\equal{\showcomments}{yes}}%
{\footnotemark
    \begin{lrbox}{\commentbox}
    \begin{minipage}[t]{1.25in}\raggedright\sffamily\small
    \footnotemark[\arabic{footnote}]}
{\begin{lrbox}{\commentbox}}}%
{\ifthenelse{\equal{\showcomments}{yes}}%
{\end{minipage}\end{lrbox}\marginpar{\usebox{\commentbox}}}
{\end{lrbox}}}

\begin{document}

\title[Maps with all nonwandering points periodic]
{On diffeomorphisms of compact 2-manifolds with all  nonwandering points periodic}

\author[S. Boyd, J.L.G. Guirao and M. Hero]{Suzanne Boyd$^1$, Juan L.G. Guirao$^2$ and Michael W. Hero$^3$}

\address{$^{1}$ Department of Mathematical Sciences,
University of Wisconsin--Milwaukee,
Milwaukee, WI 53201-0413, USA.}

\email{sboyd@uwm.edu}

\address{$^{2}$ Departamento de Matem\'{a}tica Aplicada y Estad\'{\i}stica. Universidad Polit\'{e}cnica de Cartagena,
Campus de la Muralla, 30203-Cartagena\newline (Regi\'{o}n de Murcia), Spain}
\email{juan.garcia@upct.es}

\address{$^{3}$Department of Mathematics, University of Iowa, Iowa City, IA 52242-1419}
\email{michael-hero@uiowa.edu}

\thanks{2000 Mathematics Subject Classification. Primary 37B20. Secondary 37E30.}

\keywords{periodic point, nonwandering point, Axiom A, diffeomorphism}

\begin{abstract}
The aim of the present paper is to study conditions under which all the non-wandering points  are periodic points,  for a discrete dynamical system of two variables defined on a compact manifold. We include a survey of known results in all dimensions, and study the remaining open question in dimension two. We present two results, one positive and one negative. The negative result: we construct a Kupka--Smale diffeomorphism in $\mathbb{R}^2$ (which can be extended to a diffeomorphism of the sphere)  with a closed set of periodic points that differs from the set of  nonwandering points. The positive result: we present a condition on the widely studied H\'{e}non family which guarantees that all nonwandering points are periodic. Finally, we close by describing what future work may be needed to resolve our broad goals.
\end{abstract}

\maketitle

\section{Introduction and statement of the main results}\label{Se1}

Let $\psi$ be a continuous self-map of a compact manifold $\mathbb{M}$.  
Given a point $x$ on $\mathbb{M}$, we define the \emph{trajectory} or \emph{orbit} of $x$ under $\psi$ to be the sequence $\{\psi^{n}(x)\}_{n=0}^{\infty}$, where  $\psi^{0}:=\text{Id}_{\mathbb{M}}$, the identity map on $\mathbb{M}$, and $\psi^{n}(x):=\psi^{n-1}\circ \psi$ for $n>0$. Then $(\mathbb{M},\psi)$ is a discrete dynamical system. 

\smallskip
The study of a discrete dynamical system begins with a qualitative analysis of the time evolution, including a description of the asymptotic behavior of trajectories.  One of the next natural steps is to understand or classify which dynamical systems in a given collection or family of systems will share some common orbit behaviors.

As the theory of discrete dynamical systems has developed over the last several decades, we have learned how surprising, intricate, and chaotic this behavior might be.  In this paper we are concerned with further classifying the dynamical systems whose orbits do \emph{not} exhibit chaotic behavior. 

Whether a map is chaotic depends upon the behavior of recurrent points: points whose orbits return near to themselves at some time.  The simplest type of points which exhibit some recurrent behavior are the periodic points:  a point $x$ is called \emph{periodic} if there exists a non--negative integer $p$ such that $\psi^{p}(x)=x$. If $x$ is periodic then the least $p>0$ with $\psi^{p}(x)=x$ is the  \emph{period} of $x$.   A larger class of recurrent points that play a key role are the $\emph{non-wandering points}$: a point $x \in \mathbb{M} $ is called a $\emph{non-wandering point}$ if for every open neighborhood $\mathbb{O}$ of $x$  there is an $n>0$, depending on $x$, such that $\psi^n(\mathbb{O}) \bigcap \mathbb{O} \neq \emptyset$.   We denote the set of all periodic points by $\text{P}(\psi)$, the set of periodic points of period n by  $\text{P}(\psi,n)$, and the set of non-wandering points by  $\Omega(\psi)$.

\smallskip
From the definitions we have the following chain of inclusions:
\begin{equation*}\label{eq1}
\text{P}(\psi,n)\subseteq \text{P}(\psi)\subseteq \Omega(\psi).
\end{equation*}

Our broad goal is to classify which dynamical systems of compact 2-manifolds have $P(\psi) = \Omega(\psi)$, that is, 
 the nonwandering set consists entirely of periodic points.  It is an easy exercise to show that the set of non-wandering points is closed.  Thus our motivating question is:  

\begin{question}
{If $P(\psi)$ is closed, under what additional conditions on the dynamical system $(\mathbb{M}, \psi)$ can we conclude that $P(\psi) = \Omega(\psi)$?}
\end{question}

\subsection{Dimension One}

A key fact for continuous self-maps of the interval was established by Sarkovskii in 1964 \cite{Sh}.

\begin{theorem}[Sarkovskii]
Let $ f:[0,1]\rightarrow [0,1] $ be continuous.
If $ P(f) $ is a closed set, then every periodic point of $f$ has period a power of $2$.
\end {theorem}

Next consider Sarkovskii's ordering of the appearance of periodic points.

\begin {theorem} [Sarkovskii's ordering]
Let $f:[0,1]\rightarrow [0,1]$ be a continuous function, and let the  positive integers be totally ordered in the following way:
\[3\prec 5\prec 7 \prec \cdots \prec2\cdot 3\prec 2\cdot 5\prec 2\cdot 7 \prec \cdots \prec2^2\cdot3 \prec 2^2\cdot 5 \prec 2^2\cdot 7  \prec \cdots \]
\[\cdots \prec 2^4\prec 2^3\prec 2^2\prec 2\prec 1 \]
If $f$ has a periodic point of period $n$ and $n \prec m$, then $f$ has a periodic point of period $m$.

\end{theorem}

These results restrict the class of self maps of the interval to itself with a closed set of periodic points to the non chaotic mappings.

In 1982, our main question was answered in dimension one for self-maps of a closed interval, independently by Z. Nitecki (\cite{N}) and Jin-Cheng Xiong (\cite{X}): the hypothesis needed is simply that $f$ is continuous.

\begin{theorem}[Nitecki, Xiong]
If $f$ is a continuous map of a closed interval to itself and $P(f)$ is a closed set, then $P(f)=\Omega(f)$.
\end{theorem}

It should be noted that there are self maps of the interval with every periodic point having period a power of two such that the set of all periodic points is not a closed set. This result is due to Hsin Chu and Jin Cheng Xiong \cite{CX}.

Further, Nitecki points out that $P(f)$ closed does not imply that set is finite, nor is the set of least periods necessarily finite.  His example: string together maps $f_n \colon [\frac{1}{n}, \frac{1}{n-1}] \to [\frac{1}{n}, \frac{1}{n-1}]$ so that $f_n(\frac{1}{k}) = \frac{1}{k}$ and Per$(f_n)$ contains points of least period $2^n$, but  none higher.    

However, the slightest change in the topological type of the space will cause these theorems to be false. This is easily seen by considering a circle map under an irrational rotation, where every point is non-wandering but the set of periodic points is the empty set.

Now, if we consider the dimension of the manifold to be larger than one, soon we will discover the situation is much more complicated already in dimension two.

\subsection{Dimension Two}\smallskip

If $\mathbb{M}$ is of  dimension two,  the hypothesis that the map $\psi$ be continuous does not guarantee that $P(\psi)$ closed implies $P(\psi) = \Omega(\psi)$.  In particular,
for continuous self-maps of the unit square,  Guirao and Pelayo \cite{JL} showed in 2008 that $P(\psi)$ closed does not imply that $P(\psi)$ equals $\Omega(\psi)$, although for some years this implication was considered as true in the literature, see \cite{Efre} and \cite{JL}.

Their example was a skew-product, or triangular map;  i.e., maps of the form $\psi(x,y) = (f(x), g(x,y))$, where $f(x)$ is called the \emph{base map} of the system.  Such maps form a good bridge between study of one and two variable systems.
For $\mathcal{C}^1$ skew--product self--maps defined on the unit square, Arteaga \cite{Ar} supplied a sufficient hypothesis:

\begin{theorem}[Arteaga] Suppose $\psi$ is a $\mathcal{C}^1$ skew product self-map of the unit square, $\psi(x,y) = (f(x), g(x,y))$, and  all periodic points of the base maps $f$ are hyperbolic. If $P(\psi)$ is closed, then $P(\psi) = \Omega(\psi)$.  
\end{theorem} 

\smallskip

Thus he required the map $\psi$ to be continuously differentiable, and have a base map with only hyperbolic periodic points. 

This led researchers to consider the question:

\begin{question}\label{c1}
If $\psi$ is a self--diffeomorphism of a compact 2--dimensional manifold (or $\mathbb{R}^2$) with all periodic points hyperbolic, then does $P(\psi)$ closed imply $P(\psi) = \Omega(\psi)$?
\end{question}

In our first result, we answer this question negatively. 

\begin{theorem}\label{th1} 
There exists a self--diffeomorphism $\psi_{\lambda}$ of the sphere
which has closed periodic set, and all periodic points hyperbolic, but $P({\psi_{\lambda}}) \neq \Omega(\psi_{\lambda})$.
\end{theorem}

Theorem \ref{th1} will be proved in section \ref{sec2}.

\smallskip

Seeing this negative result, it is natural to seek a stronger hypothesis. A diffeomorphism $\psi$ is called \emph{Kupka--Smale} if all periodic points are hyperbolic, and there are no tangential intersections between stable and unstable manifolds of periodic points.  Then we ask:

\begin{question}\label{c2}
If $\psi$ is a self--diffeomorphisms of $\mathbb{R}^2$, which is Kupka-Smale, then does $P(\psi)$ closed imply that $P(\psi) = \Omega(\psi)$?
\end{question}

Unfortunately, the answer is again ``no''.

\begin{corollary}\label{coro1}
The diffeomorphism $\psi_{\lambda}$ constructed in Theorem \ref{th1} is Kupka-Smale, has $P(\psi_{\lambda})$ closed but but $P(\psi_{\lambda}) \neq \Omega(\psi_{\lambda})$.
\end{corollary}

Corollary \ref{coro1} will be proved in section \ref{sec2}.

\smallskip

The obvious way to strengthen the KS hypothesis is consider \textit{Axiom A} diffeomorphisms. A diffeomorphism $\psi$ is Axiom A if it is Kupka-Smale, and there is an invariant splitting of the stable (unstable) tangent bundle which is preserved and contracted (expanded) by $D\psi$. But an Axiom A diffeomorphism always satisfies the closure of the periodic points equals the nonwandering set.  Hence this case is trivial.  Thus this hypothesis is too strong.

Thus we must seek an alternate approach to finding conditions on diffeomorphisms which satisfy that $P(\psi)$ closed implies all nonwandering points are periodic.  

Toward that end, we re-examine the dimension one case, as applied to the quadratic family $q_{a}(x) = a- x^2$.  The dimension one results say that the maps with $P(q_a)$ closed are the maps for the parameter range from $a=-1/4$, the appearance of the first fixed point, up through the period doubling cascade which occurs as $a$ increases, but not including the Feigenbaum point at the end of the cascade (that map has nonperiodic nonwandering points). [Alternatively imagine the logistic family $L_{\lambda}(x) = \lambda x (1-x)$ for $\lambda=0$ up through the period doubling cascade.] For each map in this range, there are periodic points of period $1, 2, 4, ..., 2^{n} $ for some $n$.  If $n=0$, and $a > -1/4$, there is one attracting and one repelling fixed point.  If $n\geq 1$, for parameters $a$ which are not bifurcation points,  there is an attracting cycle of period $2^n$, and repelling cycles of each period $2^j$ for $1\leq j < n$ and there are two repelling fixed points.  Such maps have all periodic points hyperbolic. For parameters $a$ which are bifurcation points, the largest period cycle is indifferent and weakly attracting. 
Thus these are maps with not all periodic points hyperbolic, but yet the periodic set is closed and equal to the nonwandering set.  Thus, this is one example which shows the hypothesis that all periodic points are hyperbolic is not necessary in dimension one.

With this example in mind, we present our second main result concentrating on a widely studied family of diffeomorphisms, the  H\'{e}non family, which generalize the quadratic family to maps of two variables.

\begin{theorem}\label{th2}
Let $H_{a,b}(x,y)=(q_{a}(x)-by,x)$, be the 2--parameter family of  H\'{e}non diffeomorphisms of $\mathbb{R}^2$, where $q_{a}(x)=a-x^2$ is the one parameter quadratic family. If $\emph{P}(q_{a})$ is closed, then for all $b$ sufficiently small, 
$\emph{P}(H_{a,b})$ is closed and equals the nonwandering set of $H_{a,b}$.
\end{theorem}

Theorem \ref{th2} will be proved in section \ref{sec3}.

\smallskip
\subsection{Further avenues of study}

In section \ref{sec4} we discuss some open questions our study has raised in two dimensions.

Finally, we note that in 1978, Danker \cite{D} showed that for manifolds of dimension at least $3$, the situation is much more complicated. 

\begin{theorem}\label{Da}[Danker]
On any manifold $\mathbb{M}$ with dimension $n\geq 3$ there exists a diffeomorphism $\psi$ whose nonwandering set is formed by hyperbolic points but
 $\overline{\emph{P}(\psi)}\neq \Omega(\psi$).
\end{theorem}

Thus, for now, we are content to concentrate on the dimension two case.

\section{Proof of Theorem \ref{th1}}\label{sec2}
In this section we shall prove Theorem \ref{th1}, Corollary \ref{coro1} and discuss the implications of this results.

\begin{proof}[Proof of Theorem \ref{th1}]

The following mapping is an example of such a diffeomorphism of the unit sphere $S^2$. Let 
$(r, \theta, \phi)$ denote the standard spherical coordinates on $\mathbb{R}^3$; i.e.,  $(x,y,z) = (r, \theta, \phi)$ means $(x,y) = (r, \theta)$ are the polar coordinates and $\phi$ is the angle the $z$-axis makes with the ray  from the origin to $(x,y,z)$.
Then  $S^2 ={\{(1,\theta, \phi)| \ \theta \in \mathbb{R} , \phi \in [0,\pi] \}}$. 
Define $\psi_\lambda:S^2 \rightarrow S^2$ as follows: 

$$
\psi_\lambda(1,\theta,\phi)=(1, \ \theta+\lambda, \ \frac{4}{\pi^2}\left(\phi-\frac{\pi}{2}\right)^3 + \frac{\pi}{2}). 
$$ 

\medskip

It is easy to see that the equator (where $\phi = \pi/2$) is an invariant set, on which the map is simply a rotation by $\lambda$, and both poles are fixed (the north pole is $(1,\theta, 0)$ and the south pole is $(1, \theta, \pi)$).  Also note both poles are repelling (sources).

 If $\lambda$ is not a rational multiple of $\pi$ then on the equator, the map is simply an irrational rotation. So here we have $P(\psi_{\lambda})$ is just the two poles, hence finite and hyperbolic, but $\Omega(\psi_{\lambda} )$ also includes a copy of $S^1$, the equator.  Thus $P(\psi_{\lambda}) \neq \Omega(\psi_{\lambda} )$. 
\end{proof}

\begin{remark}
The map constructed in the proof of Theorem \ref{th1} is not Axiom A because the equator is an invariant set, but not a hyperbolic invariant set.  
\end{remark}

\begin{proof}[Proof of Corollary \ref{coro1}]

The map ${\psi}_{\lambda}$ constructed in the proof of Theorem \ref{th1} is Kupka--Smale because there are no tangential intersections between stable and unstable manifolds. Indeed, the south pole is a source, with unstable manifold the southern hemisphere, minus the equator.  The north pole is also a source, with unstable manifold the northern hemisphere minus the equator.  Since there are no other periodic points, there are no tangential intersections, since there are no intersections at all. So the map is KS.

\end{proof}

\section{Proof of Theorem \ref{th2}}\label{sec3}
Now we prove Theorem \ref{th2}.

\begin{proof}[Proof of Theorem \ref{th2}]
Let $s_1=-1/4$ be  the parameter value at which the first fixed point of the quadratic family $q_a(x) = a-x^2$ appears, hence $s_1$ is the bifurcation point after which the attracting fixed point is born.  Call $f_{2^n}$ the bifurcation parameter value after which the attracting $2^n$--cycle is born. (Note this is potentially troublesome notation.  Rather than $f_c$ denoting a map, here $f_k$ denotes a parameter value which specifies a map $q_{f_k}$.  We are using the same notation as the paper \cite{HW} in order to facilitate the insterested reader following up on this source.)

Call $F_1$ the parameter value at the climax of the period doubling cascade, so $F_1$ is the Feigenbaum parameter.  So for the quadratic family, we know $s_{1} < f_{2} < f_{4} < ... < F_{1}$,  and there are no bifurcations between any of those parameter values.

\smallskip

According to \cite[page 63]{HW}, the bifurcation points $s_{1}$, $f_{2}$,..., $f_{2^n}$,...  extend to curves in the $(a,b)$ plane, at least for $b$ small, which do not intersect each other, and which do not intersect any other bifurcation curves.    Further, for each $a$ value, if $b$ is sufficiently small, the nonwandering set of $H_{a,b}$ is identical to the nonwandering set of $q_a$.  Hence, for this range, the nonwandering set consists precisely of the finite set of periodic points of periods $1$ through $2^n$ for some $n$.  
\end{proof}

\begin{remark}
Note the  H\'{e}non maps that satisfy the conclusion of Theorem \ref{th2} have the exact same structure to their periodic set as the quadratic interval maps that satisfy $P=\Omega$, thus some are Axiom A, and some are not. 
\end{remark}

\section{Final comments and open problems}\label{sec4}

The study of  H\'{e}non maps started in section~\ref{sec3} leads to a natural next step for investigation.  
In the  H\'{e}non family there exist additional bifurcations as $b$ increases.  That is, as $b$ increases between the extended period doubling bifurcation curves of $a$, a bifurcation point is hit at which a tangential intersection of stable and unstable manifolds occurs.  We know that just past that point, there is a transverse intersection, thus a horseshoe is created, hence there are non-wandering points.  But we do not know whether the periodic set is equal to the nonwandering set for those boundary parameter values which contain only periodic points of period $2^n$ and a homoclinic tangency.  This issue raises a couple of questions.

\begin{question}
Are there any other  H\'{e}non maps with closed periodic set?  If so, for these maps, is the periodic set equal to the nonwandering set?
\end{question}

\begin{question}
If there exists a homoclinic tangency of stable and unstable manifolds of a periodic point, does that imply the periodic set is not closed?
\end{question}

\smallskip

The counterexample in section~\ref{sec2} of the map of the sphere with finite and hyperbolic periodic set but nonperiodic nonwandering points suggests a direction for further inquiry, including possibly a shift in the type of question asked.  In that example, the nonwandering set consisted of a disjoint union of a closed periodic set and an invariant domain on which the map was an irrational rotation. This leads us to wonder whether a rotation domain is the only possible obstruction. That is:

\begin{question}
If $\psi$ is a diffeomorphism of a compact 2-manifold with closed periodic set and no rotation domains, is $P(\psi) = \Omega(\psi)$?
\end{question}

Finally, we stated our broad goal as classifying maps $\psi$ with $P(\psi) = \Omega(\psi)$.  Because these are the nonchaotic mappings, those are also the mappings with zero topological entropy.  So we could re-cast our goal as:

\begin{problem}
Classify self-diffeomorphisms of compact 2-dimensional manifolds having zero topological entropy (and describe their nonwandering sets).
\end{problem}

This is an ambitious project.  Franks and Handel have recent work on classification of area preserving zero entropy maps of the sphere, see \cite{FH}.  
In order to guarantee zero topological entropy, we would restrict our study to maps which are isotopic to the identity, with a finite set of periodic points.  A broad goal would be to describe the nonwandering set for those maps (in terms of periodic points, rotation domains, and any other potential invariant sets).

\section*{Acknowledgements}
Much of this work was done during a research stay of the second author at University Wisconsin--Milwaukee, the support of the institution is greatly appreciated. We would also like to thank Kamlesh Parwani for a helpful conversation regarding references.  

This work has been partially supported by MICINN/FEDER grant
number MTM2011--22587. 



\begin{thebibliography}{99}

\bibitem{Ar} {\sc C. Arteaga},
\textit{Smooth Triangular Maps of the Square with Closed Set of Periodic Points}, J. Math. Anal. Appl. \textbf{196} (1995), 987--997.

\bibitem{BC} {\sc L.S. Block and W.A. Coppel},
\textit{Dynamics in One Dimension}, Springer Monographs in Mathematics,
Springer--Verlag, 1992.

\bibitem{C} {\sc C. Conley},
{\em Isolated invariant sets and the Morse index}, C.B.M.S. Regional Conf. Series in Math., no. {\bf 38}. Amer. Math. Soc., Providence, R.I., 1978.

\bibitem{CX} {\sc H. Chu and J.C. Xiong},
{\em A counterexample in dynamical systems of the interval}, Proc. Amer. Math. Soc. \textbf{97} (1986), no. 2, 361--366.

\bibitem{D} {\sc A. Dankner},
{\em On Smale's Axiom A Dynamical Systems}, Ann. of Math. \textbf{107} (1978), 517--533.

\bibitem{Efre} {\sc L. S. Efremova}, 
{\em On the nonwandering set and the center of triangular maps with closed set of periodic points in the base}, Dynamical Systems and Nonlinear Phenomena, Inst. Math. NAS Ukraine, Kiev, 1990, 15 -- 25 (in Russian).

\bibitem{FH} {\sc J. Franks and M. Handel},
{\em Entropy zero areas preserving diffeomorphisms of {$S^2$}}, Geom. Topol., textbf{16} (2012), no. 4, 2187--2284.

\bibitem{HW} {\sc P. Holmes and D. Whitley},
{\em Bifurcations of one and two-dimensional maps}, Phil. Trans. R. Soc. Lond. A \textbf{311} (1984), 43--102.

\bibitem{N} {\sc Z. Nitecki},
{\em Maps of the interval with closed periodic set}, Proc. Amer. Math. Soc. \textbf{85(3)} (1982), 451--456.

\bibitem{FS} {\sc J Franke and J. Selgrade},
{\em Hyperbolicity and chain recurrence}, J. Differential Equations \textbf{26} (1977), 27--36.

\bibitem{JL} {\sc J.L.G. Guirao and F.L. Pelayo},
{\em On skew--product maps with base having closed set of periodic points}, Int. J. Comp. Math. \textbf{85(3-4)} (2008), 441--445.

\bibitem{Sh} {\sc O.M. Sarkovskii},
{\em Co-existence of cycles of a continuous mapping of a line onto itself}, Ukranian Math. Z. \textbf{16} (1964), 61--71.

\bibitem{S1} {\sc S. Smale},
{\em Differentiable dynamical systems}, Bull. Amer. Math. Soc. \textbf{73} (1967), 797--817.

\bibitem{X} {\sc J.C. Xiong},
{\em {$\Omega (f\,\mid \,\Omega (f))=P(\bar f)$} for a continuous self-mapping {$f$} of an interval}, Kexue Tongbao (Chinese), Chinese Science Bulletin, \textbf{27} (1982), no. 9, 513--514. 

\end{thebibliography}
\end{document}